\documentclass[11pt]{amsart}
\usepackage{geometry}                % See geometry.pdf to learn the layout options. There are lots.
\usepackage{graphicx}
\usepackage{amssymb}
\usepackage{epstopdf}
\usepackage{framed}
\usepackage{dsfont}
\usepackage{hyperref}
\usepackage[nameinlink,noabbrev]{cleveref}
\usepackage[normalem]{ulem}
\usepackage{todonotes}

\definecolor{skyblue}{rgb}{0.85,0.85,1}

\newtheorem{theorem}{Theorem}[section]
\newtheorem{lemma}[theorem]{Lemma}

\theoremstyle{definition}

\theoremstyle{remark}
\newtheorem{rem}[theorem]{Remark}

\newcommand{\cL}{\mathcal{L}}
\newcommand\cO{\mathcal{O}}

\newcommand\CC{\mathds{C}}

\newcommand{\bbH}{\mathds{H}}
\newcommand{\bbI}{\mathds{I}}
\newcommand\R{\mathds{R}}
\newcommand{\Z}{\mathds{Z}}

\newcommand{\ve}{\varepsilon}
\newcommand{\q}{\theta}

\DeclareMathOperator{\re}{Re}

\DeclareMathOperator{\ind}{ind}
\DeclareMathOperator{\sige}{\sigma_\textrm{ess}}

\title[Stability of Fisher-Stefan Waves]{Stability of Asymptotic Waves in the Fisher-Stefan Equation}
\author{T. T. H. Bui, P. V Heijster, \& R. Marangell}
\date{\today}                                           % Activate to display a given date or no date

\begin{document}
\begin{abstract}
We establish spectral, linear, and nonlinear stability of the {\color{black}{vanishing}} and slow-moving travelling waves that arise as time asymptotic solutions to the Fisher-Stefan equation. Nonlinear stability is in terms of the limiting equations that the asymptotic waves satisfy. 
\end{abstract}

\maketitle

\parskip=-0.5ex{\tableofcontents}
\parskip=1ex

\section{Introduction}

The Fisher-Stefan equation is the Fisher-Kolmogorov-Petrovskii-Piskunov (F-KPP) nonlinear partial differential equation (PDE) with a Stefan moving-boundary condition (\cref{eq:movingb} below): 
\begin{subequations}\label{eq:fkppstef}
\begin{align}
u_t &= d u_{xx} + u(a-bu), \qquad && t >0, \quad 0< x <h(t), \\ 
u_x(t,0) &=  u(t,h(t)) = 0, \quad && t>0, \\ 
h'(t) & = -\mu u_x(t,h(t)), \quad && t>0, \label{eq:movingb} \\ 
h(0)  &= h_0, u(0,x) = u_0(x), \quad && 0\leq x\leq h_0. 
\end{align}
\end{subequations}
Here, subscripts of the independent variables indicate partial derivatives, $x = h(t)$ is the moving boundary condition to be determined, the positive constants $d,a,b,h_0$ and $\mu$ are given, and the initial function satisfies
\begin{equation}\label{eq:cauchy}
u_0 \in C^2([0,h_0]), \quad u_0'(0) = u_0(h_0) = 0, \quad \textrm{and } \quad u_0 >0 \quad \textrm{in } [0,h_0).
\end{equation}
For our exposition, we will set without loss of generality $d = a = b = 1$, but all of our results are applicable to the general case with the appropriate modifications.

The Fisher-Stephan equation, \cref{eq:fkppstef} has been introduced and studied in \cite{DL10} as a means to model a population {\color{black}{with finite support}} undergoing logistic growth with spatial diffusion, but which could also possibly go extinct. It has been subsequently studied in \cite{DG12}, \cite{DL15},\cite{EM19}, \cite{EMS21} and  \cite{MES22}, as well as in the review article \cite{SM24}. In \cite{DL10}, the existence of vanishing solutions {\color{black}{(i.e. solutions that go to zero in the limit $t \to \infty$)}}, as well as solutions travelling with wave speeds $0\leq c<2$ was established for \cref{eq:fkppstef} with initial conditions in \eqref{eq:cauchy}.

The results of \cite{DL10} are somewhat surprising because, for the classical Fisher-KPP boundary value problem on the line:
\begin{equation} \label{eq:fkppline}
\begin{split}
u_t  = u_{xx} + u(1-u), \qquad
u(t,-\infty) = 1\,, \quad \textnormal{and} \quad u(t,+\infty) = 0\,,
\end{split}
\end{equation}
both the zero solution and travelling wave solutions of speed $0\leq c < 2$ are classically known to be (absolutely) unstable (see for example, \cite{satt77} or the review article \cite{saarloos03}, or some of the authors' own work on the subject \cite{HvH15} for more details). However, as we shall show, the adaption of the problem to one with a moving boundary leads to asymptotic problems which admit {\color{black}{these}} spectrally stable solutions. Nonlinear stability of these solutions relative to perturbations in the appropriate spaces will then follow from spectral stability and some general results from the theory of dynamical systems on Banach spaces.

In \cite{DL10} it was shown that asymptotically (as $t\to \infty$), a so-called `spreading-vanishing dichotomy' occurs for solutions to the Cauchy problem in \cref{eq:fkppstef} with initial data given by \eqref{eq:cauchy} with a small enough initial profile. Namely, solutions to \cref{eq:fkppstef} with small enough positive initial data would approach either $u=0$, or a fixed travelling wave with speed $0\leq c<2$ arising as the solution to an ordinary differential equation (ODE) boundary value problem. 
The nature of the asymptotic solution, in addition to depending on the size of the initial profile, depends on the length of the initial interval $h_0$. In particular, for the assumed parameter values $d = a = b = 1$ the results of~\cite{DL10} imply that if $h_0 \leq \pi/2$\, then the solution to \cref{eq:fkppstef} tends, as $t \to \infty $, to $u \to 0$ (i.e. the vanishing case), and $h(t) \to h_\infty = \pi/2$. In contrast, if $h_0 \geq \pi/2$, then $h(t) \sim ct$ (that is $h$ tends to a linear function of $t$ as $t \to \infty$) and the solution to \cref{eq:fkppstef} tends to the solution of the following ODE boundary value problem on the half line in the travelling wave coordinate $z := x-ct$
\begin{equation} \label{eq:fishhalf}
\begin{split}
u_{zz} + c u_z + u(1-u) = 0\,, \qquad
u_z(-\infty) = 0, \quad \textrm{and } \quad u(0) = 0.
\end{split}
\end{equation}
Furthermore, it was shown that in the spreading case, the limit of the solution to the Cauchy problem \cref{eq:fkppstef} with initial data \eqref{eq:cauchy} converges  to 1 uniformly in any compact interval of $[0,\infty)$. 
The wave speed $c<2$ is (uniquely) determined by the positive parameter $\mu$ and the boundary condition $h'(t) = -\mu u_x(t,h(t))$ for $t>0$. 
To the best of our knowledge, a general formula for determining the wave speed from the parameter $\mu$ is not known, but some estimates have appeared in \cite{DL10}, which were updated in \cite{BDK12}, and an asymptotic formula that is accurate for small values of $c$ was derived in \cite{EM19}. {\color{black} In \cref{appendix} we derive a power series expansion of the unstable manifold of the saddle in the phase plane of \cref{eq:fishhalf} at $(0,1)$, viewed as a graph over the stable manifold, which we can then use to approximate the value of $\mu$ as a function of $c$ for all values of $0<c< 2$.}

In this paper we examine the limiting PDEs, and determine the spectrum of the linearised operators, linearised about either the vanishing or the spreading solution. We find that under cases compatible with the results of \cite{DL10}, these solutions are spectrally stable in their respective PDEs. Then, standard results from the theory of dynamical systems in Banach spaces allow us to deduce nonlinear stability relative to perturbations in an appropriately chosen Banach space. 

The paper is organised as follows, in \cref{sec:vanishing}, we show that the asymptotically vanishing solution is nonlinearly stable in the PDE boundary value problem describing it. In \cref{sec:spreading}, we show that the spreading solution is spectrally stable relative to the PDE boundary value problem that describes it. 
The calculation in \cref{sec:point} shows that the angular coordinate $\q(z;\lambda)$ is, for each fixed $z \in (-\infty,0]$, a monotone function of $\lambda$. This is essentially a Maslov index calculation, though in this low dimensional setting more elementary means are sufficient to establish the monotonicity argument. Our argument is the one from Chapter 8 of \cite{codlev55}, adapted to the half-infinite interval.
In \cref{sec:nonlinear}, we apply some known results from the theory of dynamical systems to show nonlinear stability of the asymptotic solutions. We conclude the manuscript with some discussion in \cref{sec:disc}.

\section{Vanishing solution}\label{sec:vanishing}
We first show that the vanishing solution is spectrally stable. We consider a perturbation of the zero solution of the asymptotic boundary value problem:
\begin{equation} \label{eq:van}
\begin{split}
u_t  = u_{xx} + u(1-u)\,, \qquad
u_x(t,0)   =0\,, \quad \textnormal{and} \quad u(t,h_\infty) = 0.
\end{split}
\end{equation}
Here, we consider $h_\infty$ as the asymptotic size of the domain as $t\to \infty$.  
Substituting $u = 0 + \ve p(t,x)$ and keeping terms of order $\ve$ we are led to the linear boundary value problem 
\begin{equation}\label{eq:linvan}
\begin{split}
p_t  = p_{xx} + p\,, \qquad 
p_x(t,0)  =0\,, \quad \textnormal{and} \quad p(t,h_\infty) = 0,
\end{split}
\end{equation}
which is solvable by elementary PDE methods. Separation of variables yields the eigenfunctions 
$$
\phi_n(x) = \cos \left( \dfrac{(2n-1)\pi x}{2h_\infty} \right)\,, \quad n \geq 1 \in \Z\,,
$$
with simple eigenvalues 
$$
\lambda_n = 1 - \dfrac{(n-\frac12)^2 \pi^2}{h_\infty^2}.
$$
As the eigenvalues are a decreasing set, spectral stability of the zero solution as a solution of \cref{eq:linvan} follows provided $\lambda_1 \leq 0$. Substituting in $n=1$ and setting $\lambda_1 = 0$ and rearranging yeilds precisely the condition that $h_\infty = \pi/2$. We note that this {\em a fortiori} explains the value of $h_\infty = \pi/2$, and subsequently the critical $h_0$ for the vanishing problem~\cite{DL10}: it is the largest possible value on which a stable zero solution (in the sense that the solution is stable in the limiting PDE that it satisfies) can possibly occur. Hence $[0,\pi/2]$ is the largest interval that will admit solutions that vanish. For any larger $h_\infty$, we see that the zero solution is spectrally (and by standard results in the literature, linearly and nonlinearly \cite{Hen80}) unstable, and so no vanishing can occur on an interval longer than $\pi/2$. 

Nonlinear stability of the zero solution in \cref{eq:van} follows from standard results in the literature \cite{Hen80}. Namely, the linearised operator $L := \partial_{xx} + 1$, with boundary conditions $p'(0) = p(h_\infty) = 0$ is self-adjoint, bounded from above and has a compact resolvent. This together with the fact that the spectrum is contained in the left half plane, save for potentially a simple eigenvalue at $\lambda = 0$ (and only when $h_\infty = \pi/2$) means that we have nonlinear stability of the zero solution for suitably chosen perturbations \cite{Hen80}. We defer a more detailed discussion of nonlinear stability to \cref{sec:nonlinear} where we will derive nonlinear stability of the spreading solution in terms of its asymptotic PDE as well.

\section{Spreading solution} \label{sec:spreading}
We next turn our attention to the spreading solutions. Passing to a moving coordinate frame by setting $z = x-ct$, and $\tau = t$, it was shown in \cite{DL10} that $h \sim c \tau $ as $t \to \infty$ and solutions to \cref{eq:fkppstef} with Cauchy data in accordance with \eqref{eq:cauchy} will tend to a standing wave (i.e. $\tau$ invariant) solution of the following problem
\begin{equation}\label{eq:trav}
\begin{split}
U_\tau &=  U_{zz} + c U_z+ U(1-U)\,, \qquad \tau >0, \quad -\infty< z <0\,, \\ 
U_z(\tau,&-\infty) =  U(\tau,0) = 0 \quad t>0\,, \\ 
\end{split}
\end{equation}
where 
\begin{equation}\label{eq:asymptoticspeed}
c = h'(\tau)  = -\mu U_z(\tau,0).
\end{equation}
Such a solution will be a solution $u(z)$ to the following ODE boundary value problem 
\begin{equation} \label{eq:twode}
\begin{split}
u_{zz} + c u_z + u(1-u) = 0\,, \qquad
u_z(-\infty) = 0\,, \quad \textrm{and } \quad u(0) = 0.
\end{split}
\end{equation}
There exists a unique positive, (for $z \in (-\infty,0]$) decreasing solution to \cref{eq:twode} whenever $0\leq c <2$. One way to realise this well-known solution is to first define $v:=u_z$ and then the asymptotic spreading solution of \cref{eq:fkppstef} can be realised as the part of the unstable manifold in the fourth quadrant of the phase plane (i.e. with $u\geq0$ and $v\leq 0$)  of the saddle point $(1,0)$ of the following system
\begin{equation}\label{eq:nonlinsys}
\begin{pmatrix}
u \\ v 
\end{pmatrix}_z  = F(u,v) := \begin{pmatrix} v \\ -cv -u(1-u) \end{pmatrix}.
\end{equation}
That such a {\color{black}{positive}} solution is unique follows from the hyperbolicity of the critical point $(1,0)$ of \cref{eq:nonlinsys} and standard results in invariant manifold theory (see, e.g. \cite{mei07}). We note that unlike the counterpart on the line, this solution is not translation invariant, as the boundary condition explicitly prescribes the value that the solution must take at $z =0$. Thus, if we denote the solution travelling with speed $c$ by $\bar{u}_c$ (with some apologies for abusing notation as that this does not mean differentiation with respect to $c$) we have that, for example, the solution to the standing wave equation ($c=0$) is given as
\begin{equation}
\label{eq:u0}
\bar{u}_0(x) = 1-\frac{3}{1+2 \cosh{(x)}-\sqrt{3}\sinh{(x)}}\,.
\end{equation}
See figure~\ref{fig:solutions} for a plot of $\bar{u}_0(x)$ \eqref{eq:u0} and the corresponding solution of \cref{eq:nonlinsys}.
\begin{figure} \
\includegraphics[scale=0.70]{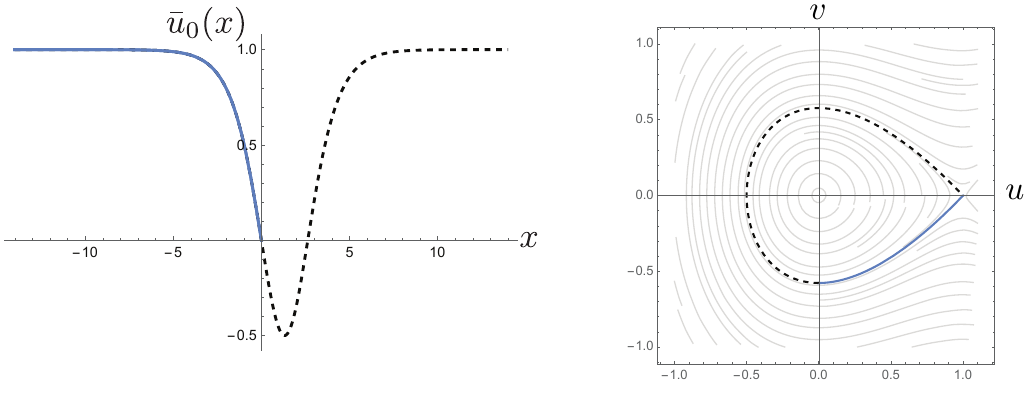} 
\caption{Left: a plot of the limiting solution $\bar{u}_0(x)$ \eqref{eq:u0} on $(-\infty,0]$ (solid line - blue online). The black dashed line illustrates how it compares to a standing wave solution of the classical Fisher-KPP equation on the line. Right: the corresponding phase portrait of \cref{eq:nonlinsys} with $c=0$. The solution (solid blue line) to the Fisher-Stefan equation is realised as the part of the unstable manifold of the saddle point $(1,0)$ of \cref{eq:nonlinsys} in the fourth quadrant.} \label{fig:solutions}
\end{figure} 

Returning to spectral stability, as above we denote the solution to \cref{eq:twode} as $\bar{u}_c$, and similarly, the stable manifold {\color{black}{of interest}} in quadrant IV of \cref{eq:nonlinsys} as $(\bar{u}_c,\bar{v}_c)$. We consider perturbations of the form $U = \bar{u}_c+ \ve e^{\lambda \tau}p(z)$ and, substituting into \cref{eq:trav} and collecting terms of $\cO(\ve)$, we see that the perturbation $p$ must satisfy the following linear ODE boundary value problem:
\begin{equation}\label{eq:bvp}
\begin{split}
\cL p := p_{zz} & + c p_z + (1 - 2\bar{u}_c)p = \lambda p\,, \qquad
p_z(-\infty)= 0\,, \quad \textnormal{and} \quad p(0) = 0.
\end{split}
\end{equation}
We now establish that the spectrum of the operator $\cL$ will be contained in the left half plane.

\subsection{Essential spectrum} 
We first compute the essential spectrum of the operator $\cL$ which will follow from general results. For our definitions and setup we follow \cite{KP13}.

 The {\em resolvent} set of $\cL,$ denoted $\rho(\cL),$ is the set of $\lambda \in \CC$ such that $(\cL-
\lambda)^{-1}$ exists and is a bounded operator. The 
{\em spectrum} of $\cL,$ denoted $\sigma(\cL) := \CC\setminus \rho(\cL)$ is defined as the complement of the resolvent set in $\CC$. We say an operator $\cL$ is {\em Fredholm} if it has a finite dimensional kernel 
and a closed, finite codimensional range. We define {\em the Fredholm index} (or {\em index} when the context is clear) of $\cL$, denoted $
\ind{(\cL)}$, as the quantity $\ind{(\cL)}:= \dim{[\ker{(\cL)}]} - \dim{[\ker{(\cL^a)}]},$ {\color{black}{where $\cL^a$ is the {\em 
adjoint} of the operator $\cL$. 
The {\em essential spectrum} of a Fredholm operator $\cL$, $\sige{(\cL)}$, is the 
set of $\lambda \in \CC$ such that either $\cL - \lambda$ is not Fredholm or $\cL-\lambda$ is Fredholm but has non-zero index. 

A description of the essential spectrum follows from an application of Weyl's theorem, which says that the essential spectrum of an operator is the same up to relatively compact perturbations of the operator. We call the operator $\cL_\infty$ a {\em relatively compact perturbation} of an operator $\cL$ if $(\cL-\cL_\infty)(\cL_\infty-\lambda)^{-1}$ is compact for some $\lambda \in \rho(\cL_\infty)$. We have the following:

\begin{theorem}[Weyl's essential spectrum theorem (Theorem 2.26 from \cite{KP13})] Let $\cL$ and $\cL_\infty$ be closed linear operators on a Hilbert space. If $\cL$ is a relatively compact perturbation of $\cL_\infty$, then the following holds
\begin{enumerate}
\item The operator $\cL - \lambda$ is Fredholm if and only if $\cL_\infty - \lambda$ is Fredholm. 
\item $\ind(\cL_\infty - \lambda) = \ind(\cL - \lambda)$
\item The operators $\cL$ and $\cL_\infty$ have the same essential spectra. 
\end{enumerate}
\end{theorem}
We next claim that since the coefficient functions of $\cL$ {\color{black}{given in \cref{eq:bvp}}} approach their end states as $z \to -\infty$ exponentially fast, $\cL$ will be a relatively compact perturbation of the resulting far-field operator. In particular, because $\bar{u}_c$ approaches its end value exponentially as $z \to -\infty$, the operator $\cL$ is a relatively compact perturbation of the far-field operator $\cL_\infty$ given by
$$
\cL_\infty p  := p_{zz} + c p_z - p\,,
$$
with boundary conditions $p_z(-\infty) = 0 =p(0)$.  Heuristically, this can be seen by the fact that $\cL_\infty$ has a bounded resolvent (see the Greens function in \cref{eq:green} below), {\color{black}{and the difference $\cL-\cL_\infty = 2(1-\bar{u}_c)$ is a function which is $2$ at $z = 0$ and which decays exponentially to zero as $z\to -\infty$}}, so the resulting operator $(\cL-\cL_\infty)(\cL_\infty - \lambda)^{-1}$ will be integration against a square integrable function and hence compact. For a rigorous proof of this fact in a general setting, we direct the reader to Lemma 3.18 in \cite{KP13}. 

Weyl's theorem then gives that the essential spectrum of $\cL$ is the same as the essential spectrum of $\cL_\infty$. The spectrum of $\cL_\infty$ can be found by directly inverting $\cL_\infty - \lambda$, subject to the boundary conditions. We have that $\lambda \in \rho(\cL_\infty)$ (the resolvent set of $\cL_\infty$ - see \cref{fig:essential})
when it is to the right of the dispersion relation:
\begin{equation} \label{eq:conspec}
 \lambda = -k^2 - 1 + i c k\,, \quad k \in \R.
\end{equation}
We refer to this parabola in \cref{eq:conspec} as the {\em continuous spectrum} or the {\em Fredholm border} and in this case it is the boundary of the essential spectrum. This is because along the boundary parabola, the operator $\cL_\infty$ (and because of Weyl's theorem $\cL$) is not Fredholm.

For $\lambda \in \rho(\cL_\infty)$, we have a decaying solution to the ODE $(\cL_\infty - \lambda)p = 0$ as $z \to -\infty$ and can thus invert $\cL_\infty - \lambda$, while for $\lambda$ to the left of the dispersion relation, there are no decaying far-field solutions and the operator is therefore not invertible on $L^2(\R^-)$ and indeed it has non-zero Fredholm index for this set. See \Cref{fig:essential}. Crucially, we observe that in \cref{eq:conspec} we have that the parabola and its interior are entirely contained in the left half plane, {\color{black}{see again \Cref{fig:essential}}}. This concludes the proof that the essential/continuous spectrum of the spreading solution is contained in the left half plane.

\begin{figure}
\includegraphics[scale=0.4]{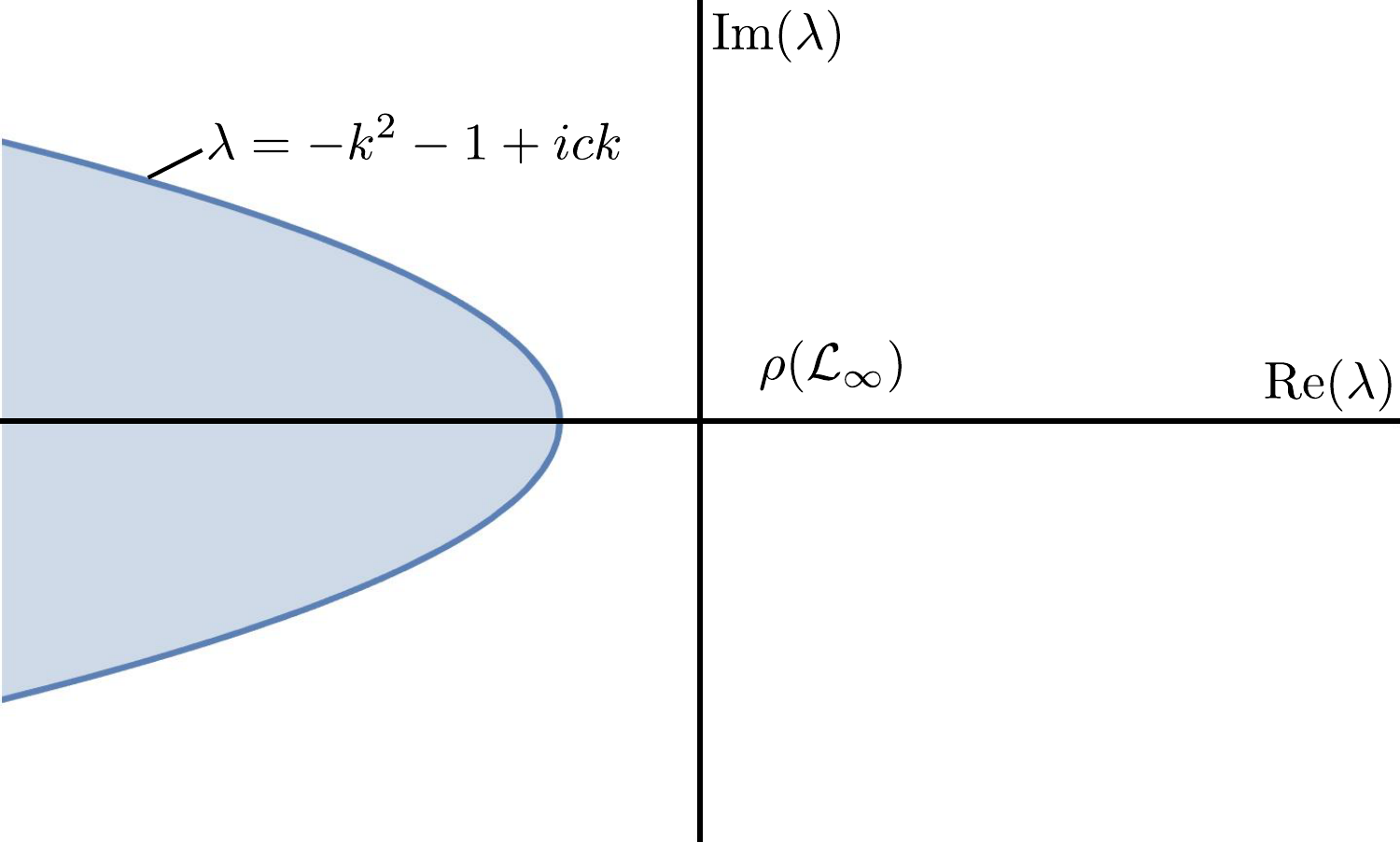}
\caption{A sketch of the essential spectrum, $\sige{(\cL_\infty)},$ (shaded region, blue online) of the operator $\cL_\infty$ and consequently of $\cL$. The operator is not Fredholm on the boundary of the essential spectrum, and has non-zero index in the interior.
It is invertible in the region to the right of the essential spectrum, $\rho(\cL_\infty)$. 
}
\label{fig:essential}
\end{figure}

For $\lambda \in \rho(\cL_\infty)$, the inverse of $\cL_\infty - \lambda$ on $L^2(\R^-)$ is given by integration against the Green's function: 
\begin{equation}\label{eq:green}
G(x;y,\lambda) :=
\begin{cases}
\dfrac{2 e^{\frac{1}{2}\left(x\sqrt{c^2+4 \lambda +4}+c(y-x)\right)} \sinh\left(\dfrac{y}{2}  \sqrt{c^2+4 \lambda+4}\right)}{\sqrt{c^2+4\lambda +4}}\,, & -\infty<x< y<0\,,  \\[4mm]
\dfrac{2 e^{\frac{1}{2} \left(y\sqrt{c^2+4 \lambda +4}+ c(y-x)\right)} \sinh \left(\dfrac{x}{2} \sqrt{c^2+4 \lambda+4}\right)}{\sqrt{c^2+4\lambda +4}}\,, & -\infty<y < x<0\,. 
   \end{cases}
\end{equation}

\subsection{Point spectrum} \label{sec:point} As the essential spectrum of $\cL$ is completely contained in the left half plane, spectral stability of the operator $\cL$ will be established by showing that there is no point spectrum of the operator $\cL$ on $L^2(\R^-)$ with $\re(\lambda) >0$. The {\em point spectrum} of the operator $\cL$ is defined as the set of $\lambda \in \CC$ such that $\cL-\lambda$ is Fredholm with index zero, but is not invertible. In our case, this will be a discrete set of values in the resolvent set of $\cL_\infty$. As such, we will sometimes refer to these $\lambda$ as {\em eigenvalues}, with the corresponding solutions as {\em eigenfunctions}. The absence of point spectrum in the right half plane will follow from the proof of the classical Sturm comparison theorem applied to a related equation, and adapted to our case. 

Recalling \cref{eq:bvp}, we set 
{\color{black} 
$q(z) := p(z) e^{\frac{c}{2}z}
$}}, and 
we have the following lemma: 
\begin{lemma}\label{lem:eq}
If $p(z)$ is a solution in $L^2(\R^-)$ to the boundary value problem given in \cref{eq:bvp} with $\lambda$ away from the essential spectrum of $\cL$, then $q$ will be a solution in $L^2(\R^-)$ to the boundary value problem 
\begin{equation}\label{eq:qeq}
\begin{split}
q_{zz}  &+ \left(1-\frac{c^2}{4} - 2 \bar{u}_c\right)q - \lambda q =0\,,  \qquad
q(-\infty) =0\,, \quad \textnormal{and} \quad q(0) = 0.
\end{split}
\end{equation}
\end{lemma}
Note, in particular, the change in the boundary condition. 
\begin{proof} Deriving the ODE is via a direct computation. 
That we can use a Dirichlet boundary condition on the left (i.e. at $-\infty$) follows from the fact that if $p$ is a solution to \cref{eq:bvp} in $L^2(\R^-)$, then it must decay to zero, and so in particular, if $p_z \to 0$, then so must $p$. In fact, $p \to 0$ as $z \to -\infty$ in only one way, namely asymptotically to decaying solutions of the far-field equation $\cL_\infty p = \lambda p$. \end{proof}

\begin{rem}We could have applied the boundary condition argument in \cref{lem:eq} directly to the boundary conditions of the perturbation in \cref{eq:bvp}, and shown that we only need Dirichlet boundary conditions there as well.\end{rem}

Applying the contrapositive of \cref{lem:eq} we have that if we can show that there are no solutions to \cref{eq:qeq}, with non-negative $\lambda$, then there will be none to \cref{eq:bvp}. 
 First, we observe that by the same considerations as for the operator $\cL$, the essential spectrum of the operator in \cref{eq:qeq} will be the set $\lambda \in (-\infty, -1-c^2/4]$. As this set is its own boundary in $\CC$, we note that it will also be the continuous spectrum. 
Next, we note that \cref{eq:qeq} is a self-adjoint eigenvalue problem, and so immediately we have that any point spectrum of the operator $\mathcal{L}$ in \cref{eq:bvp} must be real-valued.

To compute the point spectrum (or lack thereof), we reproduce the proof of the Sturm-oscillation theorem adapted to the negative half line (see for example \cite[Ch.~8, Thm~1.2]{codlev55}). We first introduce the so-called Pr\"ufer coordinates, which are polar coordinates in the dependent variables. 
We set $q(z) =: r(z) \sin(\q(z)) $ and $q_z(z) = r(z) \cos(\q (z))$ so \cref{eq:qeq} becomes:
\begin{subequations}\label{eq:bigpolar}
\begin{align}
r' & = \left(\frac{c^2}{4} + 2 \bar{u}_c + \lambda\right) \sin(\q) \cos(\q)\,, \label{eq:rad} \\
\q' & = \cos^2(\q) +\left(1 - \frac{c^2}{4}- 2\bar{u}_c - \lambda\right) \sin^2(\q) {\color{black}{= 1-\left( \frac{c^2}{4}+ 2\bar{u}_c + \lambda\right) \sin^2(\q). }}  \label{eq:polar}
\end{align}
\end{subequations}

We have the following theorem: 
\begin{theorem}[Sturm Oscillation theorem]\label{th:sturmosc}
Suppose $0<\lambda_1<\lambda_2$. We have that the Pr\"ufer angle coordinate, $\q(z;\lambda)$ satisfies:
$$
 \q(z;\lambda_2) \leq \q(z;\lambda_1). 
$$
for each $z \in (-\infty,0]$.
\end{theorem}

\begin{proof}
For a given solution $\varphi$ of \cref{eq:qeq} we have corresponding polar solutions $\varphi(z) = r(z) \sin\q(z)$ to \cref{eq:bigpolar} and further $\varphi(z)$ can only vanish when $\q(z)$ is an integer multiple of $\pi$. We also have that for $\lambda >0$, the solution which satisfies the left boundary condition satisfies
\begin{equation}\label{eq:endcond}
\lim_{z\to -\infty} \q(z;\lambda) =: \q_{-\infty}(\lambda) = \arctan\left(\frac{2}{\sqrt{c^2 + 4(\lambda + 1)}} \right)\,,
\end{equation}
which we note is a decreasing function of $\lambda$, bounded above for $\lambda \in [0,\infty)$ by its value at $\lambda = 0,$ and bounded below by $0$.

Fixing a $\lambda_2 > \lambda_1 >0$ and subtracting the solution  to \cref{eq:polar} satisfying the left boundary condition when $\lambda= \lambda_2$ from the solution to \cref{eq:polar} when $\lambda =\lambda_1$ (again, satisfying the left boundary condition) gives 
\begin{equation}\label{eq:diff}
\begin{split}
(\q(z;\lambda_1) - \q(z;\lambda_2))'  & =  (\lambda_2-\lambda_1) \sin^2 (\q(z;\lambda_2)) -\left(\frac{c^2}{4} + 2\bar{u}_c(z)\right)\left(\sin(\q(z;\lambda_1) + \sin(\q(z;\lambda_2)) \right) \\
&\qquad  \times \left( \frac{ \sin(\q(z;\lambda_1)) -\sin(\q(z;\lambda_2))}{\q(z;\lambda_1) - \q(z;\lambda_2)} \right)(\q(z;\lambda_1)-\q(z;\lambda_2)).
\end{split} 
\end{equation}
If we set $y(z)= \q(z;\lambda_1)-\q(z;\lambda_2)$, and denoting 
$$ 
m(z) := -\left(\frac{c^2}{4} + 2\bar{u}_c(z)\right)\left(\sin(\q(z;\lambda_1) + \sin(\q(z;\lambda_2)) \right) \left( \frac{ \sin(\q(z;\lambda_1)) -\sin(\q(z;\lambda_2))}{\q(z;\lambda_1) - \q(z;\lambda_2)} \right),
$$
\cref{eq:diff} becomes
$$
y' = (\lambda_2-\lambda_1) \sin^2(\q(z;\lambda_2)) + m(z)y.
$$
As $\lambda_2-\lambda_1 >0$, and $m$ is uniformly bounded, this means that we have 
\begin{equation}\label{eq:ineq}
y'(z) - m(z) y(z) \geq 0 \quad \textrm{for all } z \in (-\infty,0].
\end{equation}  
Setting $M(z) := \int_{z}^0 m(s) ds$ for $z \leq 0$, and multiplying by $e^M$, we have that \cref{eq:ineq} becomes 
\begin{equation}\label{eq:good}
\frac{d}{dz}\left(e^{M(z)} y(z)\right) \geq 0. 
\end{equation}
Integrating \cref{eq:good} on the interval $(-L,z)$, with $-L<z\leq0$, we have that 
\begin{equation}\label{eq:good2}
e^{M(z)}y(z) \geq e^{M(-L)}y(-L).  
\end{equation}
Now we choose $-L_*$ so that for all $z \leq -L_*$ we have that $\q(z;\lambda_1) > \q(z;\lambda_2)$. This is possible because the limiting function of the Pr\"ufer angle, \cref{eq:endcond} is a positive decreasing function of $\lambda$ bounded below by $0$.

Because of \cref{eq:good,eq:good2}, we have that $e^My$ is a positive, increasing function on $(-L_*,0]$, which is bounded below by $0$, and so passing to the limit we arrive at the fact that for all $z \in (-\infty,0], y(z) \geq 0$. 
\end{proof}

Now we apply \Cref{th:sturmosc} to the triple $0<\lambda <\lambda_\infty$ where $\lambda_\infty$ is chosen so that $\lambda_\infty \gg 1$, and so we have that for $\lambda \in [0,\lambda_\infty]$:
\begin{equation}\label{eq:supergood}
\q(z;\lambda_\infty) \leq \q(z;\lambda) \leq \q(z;0).
\end{equation}
The idea is to squeeze the angular coordinate between two solutions that we know. Namely, the asymptotic one for $\lambda_\infty \gg 1$, which is given by $q \sim e^{\sqrt{\lambda_\infty} z}$, and the solution to the variational equation when $\lambda = 0$. Neither of these solutions has a zero, and moreover, neither of these solutions ever has a polar angle which is a multiple of $\pi$. Moreover, for any $\lambda > \lambda_\infty$, we know that no solutions can have polar angles which are multiples of $\pi$ either (because $q \sim e^{\sqrt{\lambda} z})$. See \cref{fig:anglesandwich}. Thus, the solution satisfying the left boundary condition can never satisfy the right boundary condition of \cref{eq:qeq}, for any $z \in (-\infty,0]$, let alone at $z = 0$. We can thus conclude that there is no point spectrum to the operator $\cL$ with $\re{\lambda} \geq 0$, and so the (asymptotic) spreading solution is spectrally stable. 

\begin{figure}
\includegraphics[scale=1]{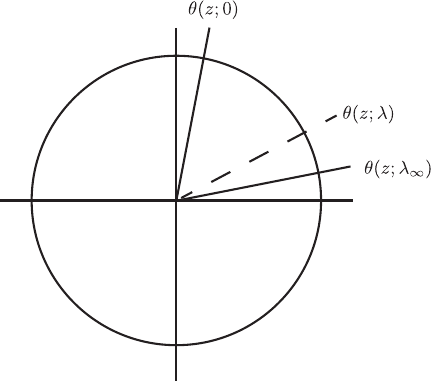}
\caption{The angular variable $\q(z;\lambda)$ for a general $\lambda >0$ is bounded for each $z \in (-\infty, 0]$ by the solutions when $\lambda = 0$ and for a large $\lambda_\infty \gg 0$. As neither of these solutions are ever an integer multiple of $\pi$, neither is the solution for any $\lambda$ in the interval $[0, \lambda_\infty]$. This leads to the conclusion that there is no point spectrum of \cref{eq:qeq}, and subsequently \cref{eq:bvp} with $\lambda \geq0$.}
\label{fig:anglesandwich}
\end{figure}

\section{Nonlinear stability} \label{sec:nonlinear}
Up until now we have been considering the spectrum of our operator $\cL$ from \cref{eq:bvp} as an operator on the Hilbert space $X = L^2(\R^-)$. However in order to discuss nonlinear stability from the framework of \cite{Hen80}, we need to restrict our perturbations to so-called {\em fractional power spaces} of the space $X$ with respect to the (sectorial) operator $\cL$. In particular, we will show nonlinear stability for perturbations in the space
$$
X^{\frac{1}{2}} = \bbH_0^1(\R^{-}) := \left\{ \phi\in \bbH^1(\R^-) \quad | \quad \phi(0) = 0 \right\}.
$$

Nonlinear stability of $\bar{u}_c(z)$ as a solution to \cref{eq:trav} in $ \bbH_0^1(\R^{-})$  will follow from combining several general results. Specifically, {\color{black}{it follows from}} the fact that our linear operator is sectorial and from Theorems 1.3.2 and 5.1.1 from \cite{Hen80}, which we reproduce in our notation for clarity. 

In \cite{Hen80} an operator $A$ is called {\em sectorial} if it is a closed, densely defined operator such that  for some $\phi \in (0,\pi/2)$ and some $M \ge 1$, the sector 
$$
S_{0,\phi} = \left\{ \lambda \quad \bigg{|} \quad |  arg (\lambda)| \leq  \phi + \frac{\pi}{2} \leq \, \lambda \neq 0 \right \} 
$$
is in the resolvent set of $A$ and (in the operator norm):
$$
|| (\lambda - A)^{-1} || \leq \frac{M}{|\lambda|}.
$$
We note here that we have flipped the signs of the sector from \cite{Hen80}, so that the resolvent is in the right half plane, however, this does not affect any of the arguments other than changing the signs where appropriate. We also note that in \cite{Hen80}, the sector need not be based at the origin, but as all of our sectors will be based at the origin, we will not need this extra flexibility. 

That $\cL$ is sectorial is due to Theorem 1.3.2 in \cite{Hen80} which says that:
\begin{theorem}[Theorem 1.3.2 from \cite{Hen80}] Suppose that $A$ is a sectorial operator {\color{black}{with sector $S$}} and $||A(\lambda - A)^{-1}|| \leq C$ for some positive $C$, and for $\lambda \in S$ and $|\lambda|$ large enough. Suppose that $B$ is a linear operator with the domain of $B, D(B)$, containing the domain of $A$, $D(A)$, and with $||B u|| \leq \ve ||A u|| + K ||u||$ for all $u \in D(A)$, and $\ve$, $K$ are positive constants with $\ve C <1$. Then $A + B$ is sectorial. 
\end{theorem}
We apply this theorem with $A = \cL_\infty$ and $B = \cL-\cL_\infty$. We have that a direct computation
from the Green's function in \cref{eq:green} gives that $\cL_\infty$ is sectorial with resolvent bound $M = 1$. Thus we have that 
$$
|| \cL_\infty (\lambda - \cL_\infty)^{-1} || \leq || \bbI + \lambda(\lambda - \cL_\infty)^{-1} ||  \leq 1+ M = 2.
$$
So we have $C =2$ in the hypotheses of the theorem. Next, we have that $B = \cL - \cL_\infty = 2(1- \bar{u}_c)$, is simply multiplication by a function which is bounded above by $2$, and so $||B u|| \leq 2 ||u|| $, so here we have $\ve = 0$, and so we can conclude that $\cL$ is sectorial as well. 

Nonlinear stability will then follow from the following theorem of Henry, which we have written in our notation: \\
\begin{theorem}[Theorem 5.1.1 from \cite{Hen80}]\label{th:nonlinstab}
Suppose that the linearisation of \cref{eq:trav} about the asymptotic spreading solution $\bar{u}_c(z)$ produces an operator $\cL$ which is sectorial and has no spectrum to the right of the line $x = \beta <0$ in the complex plane. Then $\bar{u}_c(z)$ is also nonlinearly stable as well, in the sense that there exists a $\nu$ and $M$ such that if $||V_0(z)-\bar{u}_c(z)||_{\bbH_0^1(\R^-)} \leq \nu/2M$, then there exists a unique solution $V(z,t,V_0)$  of \cref{eq:trav} existing for $t > 0 $, satisfying and with  $V(z,0,V_0) = V_0(z)$ and satisfying for all $t>0$,
$$
||V(z,t,V_0) - \bar{u}_c(z)||_{\bbH_0^1(\R^-)} \leq 2 M e^{-\beta t}||V_0(z)-\bar{u}_c(z)||_{\bbH_0^1(\R^-)}.
$$
\end{theorem}

Nonlinear stability of the zero solution in the vanishing case follows from the same considerations. In this case however, the perturbations are also required to satisfy the boundary conditions, so they must be in 
$$\bbH^1_0((0,h_\infty)) := \left\{\phi\in \bbH^1(0,h_\infty) \quad | \quad  \phi(0) = \phi(h_\infty) = 0 \right\}.$$
Indeed, the operator in \cref{eq:linvan} is self-adjoint, so the results on nonlinear stability are somewhat more classical, though they are still covered in the same framework as {\color{black}{presented here, which is based on}} \cite{Hen80}.

\section{Discussion}\label{sec:disc}
We have shown nonlinear stability relative to $\bbH^1$ perturbations satisfying the boundary conditions of the asymptotic solutions to the Fisher-Stefan problem, both in the case of vanishing solutions, as well as in the case of spreading solutions. 
We note that while we have confined ourselves to an explicit set of parameter values in \cref{eq:fkppstef} for the purposes of exposition, all of the arguments can be readily adapted, {\em mutatis mutandis}, to the more general case of arbitrary parameters. We also note that it is possible to show nonlinear stability of these operators subject to perturbations with less regularity via the use of so-called interpolation spaces (see, for example \cite{KD92}), but for many applied mathematics purposes, such as numerical simulation, $\bbH^1$ regularity is sufficient. 

We also note that the basin of attraction of the zero solution, as a solution to the PDE in \cref{eq:van} as prescribed by \Cref{th:nonlinstab}, can contain initial data which is non-positive. This is not new, indeed \cite{Hen80} explicitly computes stability of this solution in his exercises.  The spectrum moving into the right half plane due to the lengthening of the interval describes the bifurcation when the instability of zero as a solution to the logistic equation dominates the `stabilising' effects of diffusion. 

The essential ingredients of the argument in \cref{sec:point} will also work for solutions to the Fisher-KPP equation on the line, \cref{eq:fkppline}, when $c\geq2$, though again, the argument needs to be adapted to the full-line, and moreover, weighted spaces need to be considered as in that case, the essential spectrum enters into the right half plane \cite{HvH15}.  The boundary value problem in \cref{eq:twode}, in fact does not have solutions when $c\geq 2$ as then the solution satisfying the left boundary condition at $-\infty$ will never be zero. 

Our argument will not work for solutions to the Fisher-KPP equation on the line when $c<2$, which is not surprising, as the solutions corresponding to the heteroclinic orbit in the phase plane are known to be absolutely unstable. This is reflected in how the argument from \cref{sec:point} fails. Namely, the solution to the variational equation does indeed have zeros, and the polar angle will indeed wind around as $z \to \infty$. Loosely interpreting this in terms of Sturm-Liouville theory, this says that there are an infinite number of eigenvalues on the interval $\lambda >0$. This corresponds to the usual interpretation of the absolute spectrum as being the limit of the point spectrum in the large domain limit. 

\section{Acknowledgements} 
This work was partially supported by ARC DP200102130 and ARC DP210101102. RM would like to thank Maud El-Hachem, Scott McCue, Mat Simpson, and Martin Wechselberger for fruitful discussions during the writing of this paper.

\appendix
\section{A series expression for $\mu$ as a function of $c$, the wave speed} \label{appendix}

The goal of this appendix is to produce a series expansion of the parameter $\mu$ in terms of the (asymptotic) wave speed $c$. Given \cref{eq:asymptoticspeed} we have that for the standing wave solution $U(z)$ to \cref{eq:trav}
$$
c = -\mu U_z(0)\,,
$$
where $\lim_{z \to -\infty} U(z) = 1$. Working in the phase portrait of \cref{eq:nonlinsys}
$$
\begin{pmatrix}
u \\ v 
\end{pmatrix}_z  = F(u,v) := \begin{pmatrix} v \\ -cv -u(1-u) \end{pmatrix},
$$
we want a series expression for the unstable manifold of the saddle at $(1,0)$. We first shift the location of this saddle fixed point to the origin, and then shear the phase plane so that the stable eigenspace of the fixed point is the horizontal axis. Letting $(w,y)^{\top}$ be the new variables, we arrive at the new vector field
\begin{equation}\label{eq:transformedsys}
\begin{pmatrix}
w \\ y 
\end{pmatrix}_z  = H(w,y) := \begin{pmatrix} \nu w + y \\[2mm] w^2 - \dfrac{y}{\nu}\end{pmatrix}\,,
\end{equation}
where $\nu$ is the unstable eigenvalue of the saddle, given in terms of the wave speed by:
\begin{equation}\label{eq:unstableeig}
\nu := \dfrac{-c + \sqrt{ c^2 + 4}}{2}.
\end{equation}
We note that above we have used the fact that $\nu^2 + c \nu - 1 = 0$ to reduce the vector field to a form that is easier to work with below. 

Now we can write the unstable manifold as the graph of a series over the stable direction (now transformed to the horizontal axis). That is we have a series 
$$
J(w) = \sum_{i=2}^{\infty}  a_j w^j
$$
so that the curve $y = J(w)$ in the phase portrait represents a solution to the ODE. This curve, by construction, will pass through the fixed point at the origin and will be tangent to the unstable manifold there, and thus by Unstable Manifold Theorem (see for example \cite{mei07}) this curve will be the unstable manifold of the saddle. 
This allows us to write the unstable manifold of the origin in \cref{eq:transformedsys} as
$$
W_{\text{unst}} = \left\{(w,J(w)) \quad | \quad  y_z = J'(w)w_z\right\}.
$$
The equation 
$$
y_z = J'(w)w_z
$$
is called the {\em invariance condition} and guarantees that our curve in the phase portrait will be a solution to the ODE. Substituting in $H(w,y)$ for $w_z$ and $y_z$ and $J(w)$ for $y$, we arrive at an expression only in terms of $w$ and the known unstable eigenvalue $\nu$: 
$$
\left(\nu w + J(w) \right) J'(w) = w^2 - \dfrac{J(w)}{\nu}.
$$ 
This allows us to recursively solve for the coefficients of $J$ in terms of the unstable eigenvalue. The first few terms of the recurrence relation are below: 
\begin{align*}
a_2 & = \frac{\nu}{1+ 2 \nu^2} \\ 
a_3 & = \frac{-2 \nu^3}{(1+ 2 \nu^2)^2(1+3 \nu^2)} \\
a_4 & = \frac{10 \nu ^5}{(1+ 2 \nu^2)^3(1+3 \nu^2)(1+ 4 \nu^2)} \\
a_5 & = \frac{-12 \nu^7(6+19\nu^2)}{(1+ 2 \nu^2)^4(1+3 \nu^2)^2(1+ 4 \nu^2)(1+ 5 \nu^2)} \\ 
&\,\,\,\, \vdots 
\end{align*}
These can be put in terms of $c$ by using \cref{eq:unstableeig}.

In the transformed vector field, we are looking for the value of $J(-1)$, and our expression for $c$ has now becomes
$$
c = -\mu(J(-1) - \nu)\,, 
$$
or, rearranging, we get an expression for $\mu$
$$
\mu = \frac{c}{\nu - J(-1)}.  
$$

For the parameter values used in this manuscript, roughly 20 terms of the series was sufficient to approximate the intersection of the unstable manifold of the saddle at $(1,0)$ with the vertical axis to the point where the difference was visually indistinguishable. See \cref{fig:intersection}, as well as \cref{fig:cvsmu}, for numerical plots of the intersection of the unstable manifold with the vertical axis, as well as plots for $\mu$ vs $c$ and a formal inversion showing $c$ vs $\mu$. 

\begin{figure}
\includegraphics[scale = 0.25]{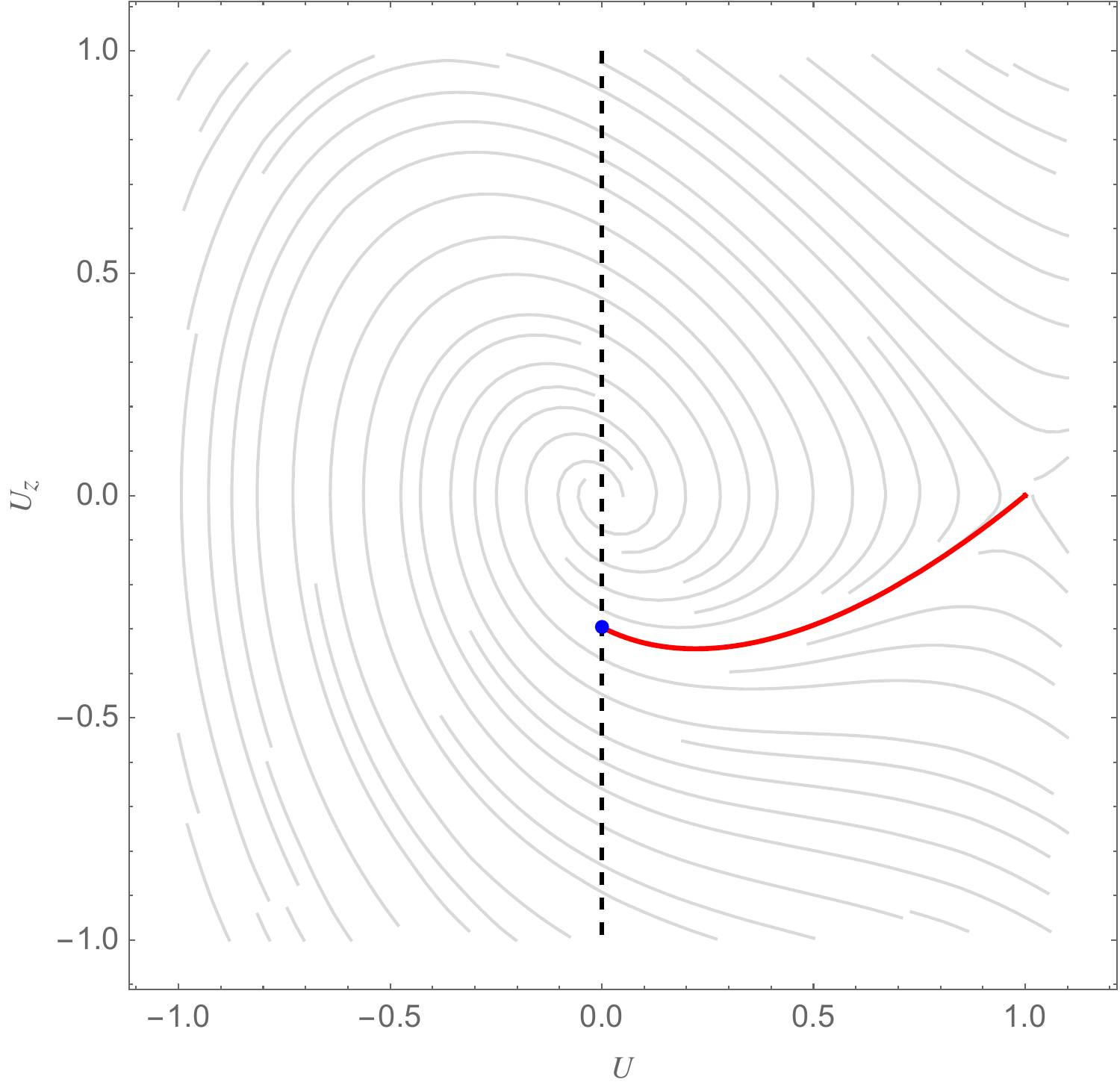} \quad
\includegraphics[scale = 0.25]{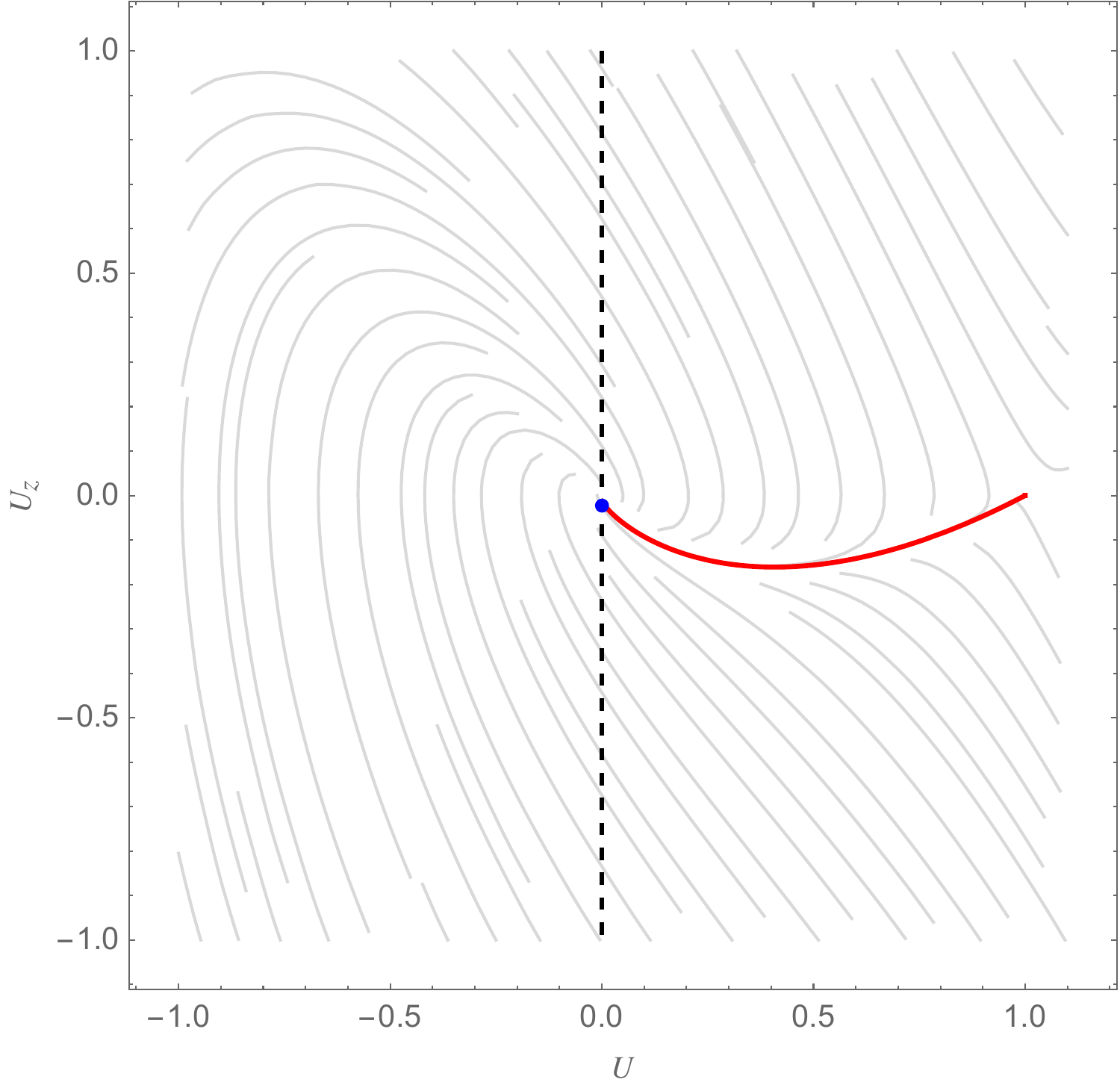}
\caption{Plots of the unstable manifold of the saddle at $(1,0)$ in the phase plane of \cref{eq:nonlinsys}. The curve (red online) intersects the vertical axis (black dashed) at the (blue online) dot. The curve is computed by numerically solving the ODE \cref{eq:nonlinsys}, while the blue dot is computed using the series expression for $J(-1)$. Left is $c=0.5$. Right is $c=1.5$ }\label{fig:intersection}
\end{figure}
\begin{figure}
\includegraphics[scale =0.5]{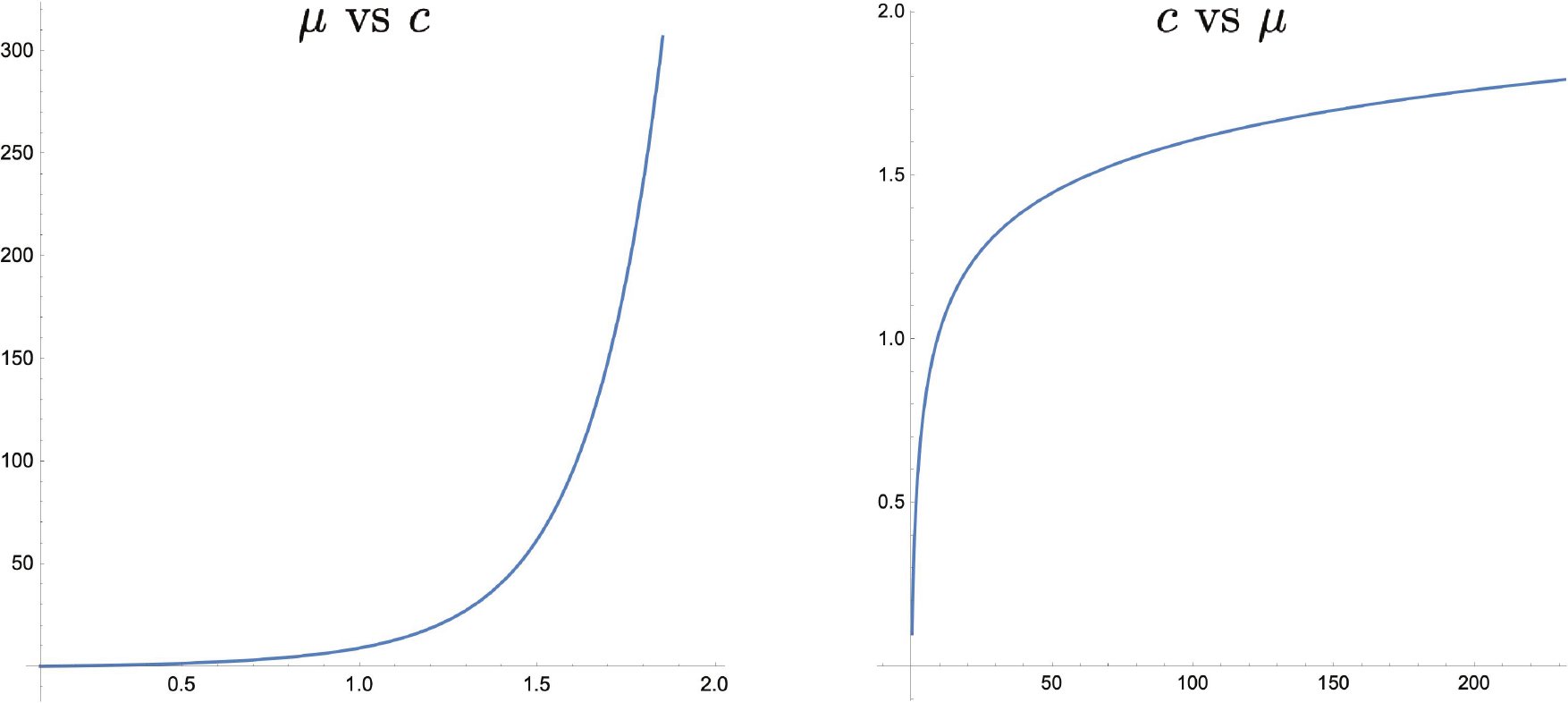}
\caption{Left: a plot of $\mu$ vs $c$ for $0<c<2$. Right: The inversion, showing $c$ vs $\mu$. }\label{fig:cvsmu}
\end{figure}

\bibliographystyle{alpha}      % mathematics and physical sciences
\bibliography{fisherkpp}  % name your BibTeX data base

\end{document}